\documentclass[a4paper]{amsart}
\usepackage{graphicx,amsmath}
\usepackage{amssymb}

\newtheorem{theorem}{Theorem}[section]

\newtheorem{lemma}[theorem]{Lemma}
\newtheorem{proposition}[theorem]{Proposition}

\theoremstyle{definition}

\theoremstyle{remark}
\newtheorem{remark}[theorem]{Remark}
\numberwithin{equation}{section}

\begin{document}

\title{On the refined shrinking target property of rotations}
\author{Dong Han KIM}
\address{Department of Mathematics Education, Dongguk University -- Seoul, Seoul 100-715, Korea}
\email{kim2010@dongguk.edu}


\begin{abstract}
We discuss the shrinking target property of irrational rotations.
We obtain the condition of an irrational $\theta$ and monotone increasing $\varphi(n)$ such that 
$$ 
\liminf_{n \to \infty} n \varphi (n) \| n\theta - s \| = 0 \text{ for almost every } s.
$$
We also consider the class of irrationals for which the limit inferior is 0 for every monotone increasing $\varphi(n)$ such that $\sum_n 1/(n\varphi(n))$ diverges.
\end{abstract}


\maketitle

\section{Introduction}

The inhomogeneous Diophantine approximation theorem by Minkowski\cite{Min} state that
for an irrational number $\theta$,
if $s$ is not of the form $B \theta - A $ for integers $A$ and $B$, then there are infinitely many integer $n$ such that
\begin{equation}\label{Minkowski} \| n \theta - s \| < \frac1{4|n|} \end{equation}
where $\| t \|$, $t \in \mathbb R$ is the distance to its nearest integer.

%

An irrational $\theta$ is said to be of bounded type if there exist a $C >0$ such that
$n \| n \theta \| > C$ for all positive integer $n$.
Kurzweil\cite{Kurzweil} showed that,
if and only if the irrational $\theta$ is of bounded type,
then for almost every $s$ and a monotone decreasing positive $\psi(n)$ with $\sum_n \psi(n) = \infty$,
\begin{equation}\label{Kurzweil}
\| n \theta - s \| <  \psi(n) \quad \text{ for infinitely many } n \in \mathbb N
\end{equation}
(See also \cite{Fa} for the higher dimensional case).
Note that the first Borel-Cantelli lemma implies that if $\sum_n \psi(n) < \infty$, then   
for almost every $s \in \mathbb R$ we have
$ \| n \theta - s \| <  \psi(n) $ holds for only finitely many $n$'s.
The refined Kurzweil type inhomogeneous Diophantine approximation has been studied in \cite{BoCh, Chaika, ChCo}.
A sequence $\psi(n)$ of positive numbers is called a Khinchin sequence\cite{ChCo} if $n \psi(n)$ is monotone decreasing and $\sum_n \psi(n) = \infty$.
In this article, we study the condition for the irrational $\theta$ and the Khinchin sequence $\psi(n)$ such that \eqref{Kurzweil} holds for almost every $s$.


The inhomogeneous Diophantine approximation is related to the shrinking target property (see \cite{Fa, IndMath2007})
and the dynamical Borel-Cantelli lemma (see \cite{CK,Dol,DCDS2007}) of the irrational rotations.
For a measure preserving transformation $T$ on $(X,\mu)$, it is proved\cite{BGI} that for $\mu $-almost all $x\in X$ one has
\begin{equation*}
\liminf_{n\geq 1}n^{\beta }\cdot d(T^{n}x,y)=\infty \text{ with }\beta > \frac{1}{\underbar d_{\mu }(y)},
\end{equation*}
where $\underbar d_{\mu }(y) = \liminf_{r \to 0} \log \mu(B(y,r))/\log r$ is the lower local dimension at $y$.
For a piecewise expanding map on an interval\cite{DCDS2007} or some hyperbolic map(\cite{CK}, \cite{Dol}) it is known that
for $\mu $-almost all $%
x\in X$ one has
\begin{equation*}
\liminf_{n\geq 1}n^{\beta }\cdot d(T^{n}x,y)= 0 \text{ with }\beta = \frac{1}{\underbar d_{\mu }(y)}
\end{equation*}

We assume that $T$ is the rotation by an irrational $\theta$ on the unit interval.
Then by \eqref{Minkowski} and Cassels' lemma\cite[Lemma 2.1]{Harman} we have
\begin{equation}\label{non}
\liminf_{n \to \infty} n \cdot \| n \theta - s \| = 0 \quad \text{ almost every } s \in \mathbb R. 
\end{equation}
(See also \cite{Non2007}).
In this paper, we consider the condition of the irrational $\theta$ and the monotone increasing $\varphi(n)$ for which
\begin{equation}\label{liminf}
\liminf_{n \to \infty} n \varphi (n) \cdot \| n \theta - s \| = 0 \quad \text{ almost every } s \in \mathbb R.
\end{equation}
For the monotone increasing $\varphi(n)$, If $\sum_n 1/(n\varphi(n)) = \infty$, then $1/(n\varphi(n))$ is a Khinchin sequence.

In Sectiion~\ref{sec2}, we state the condition of the irrational $\theta$ and the monotone increasing $\varphi(n)$ 
for which \eqref{liminf} holds (Theorem~\ref{thm}).
In Sectiion~\ref{sec3}, we give some sufficient and necessary conditions of the irrational $\theta$ that for any monotone increasing $\varphi(n)$ with  $\sum_n 1/(n\varphi(n)) = \infty$, \eqref{liminf} holds (Theorem~\ref{conditions}).
The proof of Theorem~\ref{thm} is given in Section~\ref{proof}.
The analogous result for the formal Laurent series case was studied in \cite{KimNakada}.

\section{Main Theorem}\label{sec2}

For an irrational number $0 < \theta < 1$, we have the continued fraction expansion with partial quotients $a_k$, $k \ge 1$.
Let $p_k / q_k$ be the $k$-th convergents 
with $p_0 = 0$, $q_0 = 1$. 
Then we have $q_{k+1}=a_{k+1}q_{k}+q_{k-1}$, thus
\begin{equation}\label{qkbound} q_{k+1} \ge 2 q_{k-1} \text{  for all } k\ge 1. \end{equation}

We have the main theorem of the paper as follows:
\begin{theorem}\label{thm}
Let $\varphi(n)$ be a monotone increasing positive function.
For a given irrational $\theta$  we have
$$ \liminf_{n \to \infty} n \varphi (n) \cdot \| n \theta - s \| = 0
\quad \text{ almost every } s $$
if and only if the principal convergent's denominator $q_k$ of the irrational $\theta$ satisfies 
\begin{equation}\label{condition}
\sum_{k=0}^\infty \frac{\log \left( \min ( \varphi(q_k), q_{k+1}/q_k) \right) }{\varphi(q_k)} = \infty.
\end{equation}
\end{theorem}

The proof of Theorem~\ref{thm} is given in Section~\ref{proof}.

\begin{remark}
(i) For any monotone increasing $\varphi(n)$ to $\infty$, choose an irrational $\theta$ such that $\varphi (q_k) > k^2$. Then  
$$\sum_{k=0}^\infty \frac{\log \left( \min ( \varphi(q_k), q_{k+1}/q_k) \right) }{\varphi(q_k)} \le \sum_{k=0}^\infty \frac{\log \varphi(q_k) }{\varphi(q_k)} < \sum_{k=0}^\infty \frac{2 \log k}{k^2} < \infty.$$  
Therefore, as already studied in \cite[Theorem 5]{BoCh}, 
for any monotone increasing function $\varphi(n)$ which goes to infinity, there exist an irrational $\theta$ such that 
$$ \liminf_{n \to \infty} n \varphi (n) \cdot \| n \theta - s \| = \infty
\quad \text{ almost every } s.$$

(ii) We also obtain \eqref{non}: if $\varphi(n)$ is bounded, then \eqref{condition} diverges for every $\theta$, thus for almost every $s \in \mathbb R$
$$ \liminf_{n \to \infty} n \cdot \| n \theta - s \| = 0.$$
\end{remark}

\begin{remark}
The condition $\sum_{n=1}^\infty 1/(n \varphi(n)) = \infty$ is implied by \eqref{condition}
since by defining $\varphi(x) = \varphi( \lfloor x \rfloor )$ and following proposition~\ref{qk},
\begin{equation*}
\begin{split}
\sum_{n=1}^\infty \frac{1}{n \varphi(n)} &\ge \int_1^\infty \frac{dx}{x\varphi(x)} = \int_0^\infty \frac{dt}{\varphi(e^t)}\\
&= \sum_{k=0}^\infty \left( \int_{\log q_k}^{\log q_{k+1}} \frac{dt}{\varphi(e^t)} \right) 
\ge \sum_{k=0}^\infty \frac{\log \left( q_{k+1}/q_k \right) }{\varphi(q_{k+1})} = \infty.
\end{split}\end{equation*} 
\end{remark}

\begin{proposition}\label{qk}
For a monotone increasing $\varphi(n) >0$ we have 
$$\sum_{k=0}^\infty \frac{\log \left( \min ( \varphi(q_k), q_{k+1}/q_k) \right) }{\varphi(q_k)} = \infty,$$ if and only if 
$$\sum_{k=0}^\infty \frac{\log \left( \min ( \varphi(q_k), q_{k+1}/q_k) \right) }{\varphi(q_{k+1})} = \infty.$$
\end{proposition}

\begin{proof}
It is enough to show \lq only if ' part.
Let 
$$\Lambda := \{ k \ge 0 : \varphi( q_{k+1} ) > 2 \varphi( q_{k} ) \} . $$
Then we have 
$$\sum_{k \in \Lambda }  \frac{ \log ( \min (\varphi (q_k), q_{k+1}/q_k) )}{ \varphi( q_k ) } \le \sum_{k \in \Lambda } \frac{\log \varphi (q_k) }{ \varphi( q_k ) } < \infty.$$
Therefore, if $ \sum_{k=0}^\infty \log ( \min (\varphi (q_k) , q_{k+1}/q_k) ) / \varphi( q_k ) = \infty $, then we have
$$ \sum_{k \notin \Lambda} \frac{ \log ( \min (\varphi (q_k), q_{k+1}/q_k) )}{ \varphi( q_k ) } = \infty, $$
thus
\begin{align*}
\sum_{k=0}^\infty \frac{ \log ( \min (\varphi (q_k), q_{k+1}/q_k) )}{ \varphi( q_{k+1} ) }
&\ge \sum_{k \notin \Lambda } \frac{ \log ( \min (\varphi (q_k), q_{k+1}/q_k) )}{ \varphi( q_{k+1} ) } \\
&\ge \sum_{k \notin \Lambda } \frac{ \log ( \min (\varphi (q_k), q_{k+1}/q_k) )}{ 2 \varphi( q_{k} ) }= \infty.
\end{align*}
\end{proof}

\section{Conditions for the Kurzweil type approximations}\label{sec3}

\begin{theorem}\label{conditions}
(i)
Let $\theta$ be an irrational with $q_k \le C^k$ for some constant $C$.
Then for all monotone increasing $\varphi(n) >0$ with $\sum_{n = 1}^\infty  1/(n \varphi (n)) = \infty$
\begin{equation*}
\liminf_{n \to \infty} n \varphi(n) \cdot \| n \theta - s \| = 0 \quad \text{  almost every } s .
\end{equation*}

(ii)
If there exists a constant $D$ such that $q_{k+1}/q_k \le D \log q_k$ for large $k$,
then for all monotone increasing $\varphi(n) >0$ with $\sum_{n = 1}^\infty 1/(n \varphi (n) ) = \infty$
\begin{equation*}
\liminf_{n \to \infty} n \varphi(n) \cdot \| n \theta - s \| = 0 \quad \text{  almost every } s .
\end{equation*}

\end{theorem}

The condition (i) is already discussed in \cite[Theorem 6]{ChCo}.
(see also \cite[Remark 3]{Chaika})
The condition of (ii) is not implied by condition of (i).
Let $\theta$ be an irrational with partial quotients $a_k = k$, $k \ge 1$.
Then we have 
$$ \log q_k \ge \log a_1 + \dots + \log a_k = \log 1 + \dots + \log k \ge \int_1^k \log x dx = k \log k - k + 1$$
There is no constant $C$ such that $q_k \le C^k$.
However, we have for large $k$
$$\frac{q_{k+1}}{q_{k}} = a_{k+1} + \frac{q_{k-1}}{q_k} \le k+2 \le k \log k -k +1 \le \log q_k.$$ 

\begin{proof}
By \eqref{non} we may assume that $\varphi(n)$ goes to infinity as $n$ goes to infinity.
Let $\varphi(x) = \varphi(\lfloor x \rfloor)$ be defined on real $x \ge 1$.

(i) Suppose that $q_k \le C^k$, $C >1$ and $\varphi(n)$ be monotone increasing with $\sum_n \frac{1}{n\varphi(n)} = \infty$.
Then we have
\begin{equation*}\begin{split}
\sum_{k=0}^{\infty} \frac{1}{\varphi( q_{2k})}
&\ge \sum_{k=0}^{\infty} \frac{1}{\varphi( C^{2k})} 
=\sum_{k=0}^{\infty} \frac1{2\log C}\int_{2k\log C}^{2(k+1)\log C} \frac{dt}{\varphi(C^{2k})} \\
&\ge \frac1{2\log C} \sum_{k=0}^{\infty} \left( \int_{2k\log C}^{2(k+1) \log C} \frac{dt}{\varphi(e^t) } \right)
= \frac1{2\log C} \int_{0}^{\infty} \frac{dt}{\varphi(e^t) } \\
&= \frac1{2\log C} \int_{1}^{\infty} \frac{dx}{x \varphi(x) } \ge \frac1{2\log C} \sum_{n=2}^{\infty} \frac{1}{n \varphi(n)} = \infty.
\end{split}\end{equation*}
Since for large $k$ as to $\varphi(q_{2k-1}) \ge 2$, we have by \eqref{qkbound}
\begin{multline*}
\frac{\log \left( \min ( \varphi(q_{2k-1}), q_{2k}/q_{2k-1}) \right) }{\varphi(q_{2k-1})} + 
\frac{\log \left( \min ( \varphi(q_{2k}), q_{2k+1}/q_{2k}) \right) }{\varphi(q_{2k})} \\
\ge \frac{\log \left( \min ( \varphi(q_{2k-1}), q_{2k+1}/q_{2k-1}) \right) }{\varphi(q_{2k})} \ge  \frac{\log 2}{\varphi( q_{2k})},
 \end{multline*}
we have $\sum_k \log \left( \min ( \varphi(q_k), q_{k+1}/q_k) \right) / \varphi(q_k)$ diverges.
Hence, by Theorem~\ref{thm} we complete the first claim.

(ii)
Let 
$$\alpha_k := \frac{\log (q_{k+1}/q_k)}{q_{k+1}/q_k}, \qquad \beta_k : = \int_{q_k}^{q_{k+1}} \frac{dx}{x \varphi(x)}.$$
If $\varphi(q_k) \le q_{k+1} / q_k$, then since $(\log x)/x$ is decreasing for $x > e$, for some $M$ such that $\varphi(q_M) > e$, 
we have for $k \ge M$
\begin{equation}\label{casei}
\frac{\log \left( \min ( \varphi(q_k), q_{k+1}/q_k) \right) }{\varphi(q_k)} 
= \frac{\log \varphi(q_k) }{\varphi(q_k)} \ge \frac{\log (q_{k+1}/q_k)}{q_{k+1}/q_k} = \alpha_k. 
\end{equation}
When $\varphi(q_k) > q_{k+1} / q_k$, we have
\begin{equation}\label{caseii}
\frac{\log \left( \min ( \varphi(q_k), q_{k+1}/q_k) \right) }{\varphi(q_k)} 
= \frac{\log (q_{k+1}/q_k)}{\varphi(q_k)} \ge \int_{\log q_k}^{\log q_{k+1}} \frac{dt}{\varphi(e^t)}
 = \beta_k.
\end{equation}
Therefore, by \eqref{casei} and \eqref{caseii} we have
\begin{equation}\label{akbk}
\sum_{k=M}^\infty \frac{\log \left( \min ( \varphi(q_k), q_{k+1}/q_k) \right) }{\varphi(q_k)} \ge \sum_{k=M}^\infty \min(\alpha_k, \beta_k).
\end{equation}

Since $\varphi$ is monotone increasing, for $\ell < m$ we have
\begin{equation}\label{bk}
\begin{split}
\frac{\beta_{\ell}}{\log (q_{\ell+1}/q_\ell)} 
&= \frac{1}{\log q_{\ell+1} - \log q_{\ell}} \int_{\log q_\ell}^{\log q_{\ell+1}} \frac{dt}{\varphi(e^t)} \\
&\ge \frac{1}{\log q_{m+1}- \log q_m} \int_{\log q_m}^{\log q_{m+1}} \frac{dt}{\varphi(e^t)}
= \frac{\beta_m}{\log (q_{m +1}/q_m)}.
\end{split}
\end{equation}

By the assumption $\sum_n 1/(n \varphi(n)) = \infty$, we have
\begin{equation*}
\sum_{k=0}^\infty \beta_k = \sum_{k=0}^\infty \int_{q_k}^{q_{k+1}} \frac{dx}{x \varphi(x)} = \int_{1}^{\infty} \frac{dx}{x \varphi(x)} \ge \sum_{n=2}^\infty \frac{1}{n \varphi(n)} = \infty. \end{equation*} 
Therefore, if there is only finitely many $k$'s such that $\alpha_k \le \beta_k$,
then we have for a large $N$
$$\sum_{k=N}^\infty \min(\alpha_k, \beta_k) = \sum_{k=N}^\infty \beta_k = \infty.$$

Suppose that there are infinitely many $k$'s such that $\alpha_k \le \beta_k$.
Choose the subsequence $(k_i)_{i \ge 0}$ such that $\alpha_{k_i} \le \beta_{k_i}$ and $k_i \ge M$.
Then we have
\begin{equation*}\begin{split} 
\sum_{k= M}^\infty \min(\alpha_k, \beta_k) &\ge \sum_{i=1}^\infty \sum_{k=k_{i-1}+1}^{k_{i}} \min(\alpha_k, \beta_k) 
\ge \sum_{i=1}^\infty \left( \alpha_{k_i} + \sum_{k=k_{i-1}+1}^{k_{i}-1} \beta_{k} \right) \\
&\ge \sum_{i=1}^\infty \left( \alpha_{k_i} + \sum_{k=k_{i-1}+1}^{k_{i}-1} \frac{\log (q_{k+1}/q_k) }{\log (q_{k_i+1}/q_{k_i}) } \beta_{k_i} \right)\\
&\ge \sum_{i=1}^\infty \left( \alpha_{k_i} + \sum_{k=k_{i-1}+1}^{k_{i}-1} \frac{\log (q_{k+1}/q_k) }{\log (q_{k_i+1}/q_{k_i}) } \alpha_{k_i} \right) \\
&\ge \sum_{i=1}^\infty \sum_{k=k_{i-1}+1}^{k_{i}} \frac{\log (q_{k+1}/q_k) }{\log (q_{k_i+1}/q_{k_i}) } \alpha_{k_i}
= \sum_{i=1}^\infty \sum_{k=k_{i-1}+1}^{k_{i}} \frac{\log (q_{k+1}/q_k) }{q_{k_i+1}/q_{k_i}} \\
&= \sum_{i=1}^\infty \frac{\log q_{k_i +1} - \log q_{k_{i-1}+1}}{q_{k_i+1}/q_{k_i}} 
\ge \frac 1D \sum_{i=1}^\infty \frac{\log q_{k_i +1} - \log q_{k_{i-1}+1}}{\log q_{k_i}}.
\end{split}\end{equation*}
Since for any $i < j$ such that  $ 2 \log q_{k_{i-1} +1} < \log q_{k_j +1 }$, 
\begin{multline*}
\frac{\log q_{k_i +1} - \log q_{k_{i-1}+1}}{\log q_{k_i}} + \frac{\log q_{k_{i+1} +1} - \log q_{k_{i}+1}}{\log q_{k_{i+1}}} + \dots + \frac{\log q_{k_j +1} - \log q_{k_{j-1}+1}}{\log q_{k_j}} \\
\ge 
\frac{\log q_{k_j +1} - \log q_{k_{i-1}+1}}{\log q_{k_j}} > \frac 12  \frac{\log q_{k_j +1}}{\log q_{k_j}} > \frac 12,
\end{multline*}
we have
$$\sum_{k=M}^\infty \min(\alpha_k, \beta_k) = \infty.$$
Combining with (\ref{akbk}),
we have $\sum_{k=M}^\infty \log \left( \min ( \varphi(q_k), q_{k+1}/q_k) \right) /{\varphi(q_k)} = \infty$ and 
Theorem~\ref{thm} implies that \eqref{liminf}.
Hence, we prove the second assertion.
\end{proof}


\begin{proposition}\label{prop2}
If $\theta$ is an irrational such that 
$$ \sum_{k=2}^\infty \frac{1}{\log q_k} < \infty, $$
then there is a monotone increasing $\varphi(n)$ such that 
$\sum_{n = 1}^\infty \frac{1}{n \varphi (n) } = \infty$
and 
$$ \liminf_{n \to \infty} n \varphi(n) \cdot \| n \theta - s \| = \infty \quad \text{  almost every } s . $$
However, the converse is not true.
\end{proposition}


\begin{proof}[Proof of Proposition~\ref{prop2}]
Let $\varphi(n) = \log n \cdot \log (\log n)$ for  large $n$.
Then for some $M$ we have 
\begin{equation*}\begin{split}
\sum_{k=M}^\infty \frac{\log \left( \min ( \varphi(q_k), q_{k+1}/q_k) \right) }{\varphi(q_k)} 
&\le \sum_{k=M}^\infty \frac{\log \varphi(q_k) }{\varphi(q_k)} \\
&= \sum_{k=M}^\infty \frac{\log (\log q_k) + \log ( \log (\log q_k)) }{\log q_k \cdot \log (\log q_k)} \\
&< \sum_{k=M}^\infty \frac{ 2}{\log q_k} < \infty.
\end{split}\end{equation*}
By Theorem~\ref{thm}, we complete the first assertion.

Let $\theta$ be an irrational with partial quotients
$$a_k \sim k^{\log(\log k)} \qquad \left( \text{ i.e., } \frac{k^{\log(\log k)}}{a_k} \to 1 \text{ as } k \to \infty\right).$$
Then we have
\begin{equation*} \log q_k \sim \sum_{i=1}^k (\log i )(\log(\log i)) \sim k (\log k ) (\log(\log k)),\end{equation*}
which yields
$$ \sum_{k=0}^\infty \frac{1}{\log q_k} = \infty.$$
However, if we choose $\varphi(n) = \log n \cdot \log (\log n) \cdot \log (\log (\log n)) $ for large $n$, then  we have for some $N$
\begin{equation*}\begin{split}
\sum_{k=N}^\infty \frac{\log \left( \min ( \varphi(q_k), q_{k+1}/q_k) \right) }{\varphi(q_k)} 
&\le \sum_{k=N}^\infty \frac{\log \varphi(q_k) }{\varphi(q_k)} 
\le \sum_{k=N}^\infty \frac{3}{\log q_k \cdot \log (\log (\log q_k))} < \infty 
\end{split}\end{equation*}
since 
$$
\log q_k \cdot \log (\log (\log q_k)) \sim k (\log k ) (\log(\log k))^2.
$$
Hence, the condition of $\sum_k 1/ \log q_k < \infty$ in Proposition~\ref{prop2} is not a necessary condition.
\end{proof}

\section{Proof of the main Theorem}\label{proof}

In this section, we give the proof of the main theorem.
Let $B(x,r)$ be the ball centered at $x$ with radius $r$. 
We denote $\mu$ the Lebesgue measure on the unit circle.
We assume that $\varphi(n) \ge 4$.

Denote
$$ E_k := \bigcup_{q_k < n \le q_{k+1}} B\left( n\theta, \frac{1}{n\varphi(n)} \right).$$
Then we have
\begin{equation}\label{bnek}
\bigcap_{N \ge 1} \bigcup_{n \ge N} B \left( n\theta,\frac{1}{n\varphi(n)} \right) = \bigcap_{K \ge 1} \bigcup_{k \ge K} E_k. 
\end{equation}

Since $\| n \theta - (n-q_k) \theta \| = \|q_k \theta \|$ and $\varphi(n)$ is monotone increasing, we have for each $q_{k} < n \le q_{k+1}$
$$
\mu  \left( B \left(n\theta, \frac{1}{n\varphi(n)}\right) \setminus B \left((n-q_k)\theta, \frac{1}{(n-q_k)\varphi(n-q_k)} \right) \right)  \le \| q_k \theta \|.
$$
Thus, we have 
\begin{equation*}
\begin{split}
\mu(E_{k}) &\le \sum_{n= q_k+1}^{2q_{k}} \mu \left(  B \left (n\theta, \frac{1}{n\varphi(n)} \right) \right) \\
&\quad + \sum_{n= 2q_k+1}^{q_{k+1}} \mu \left( B \left(n\theta, \frac{1}{n\varphi(n)}\right) \setminus B \left((n-q_k)\theta, \frac{1}{(n-q_k)\varphi(n-q_k)} \right) \right) \\
&\le \sum_{n= q_k+1}^{2q_{k}} \frac{2}{n\varphi(n)} + \sum_{n= 2q_k+1}^{q_{k+1}}  \min \left(\| q_k \theta \|, \frac{2}{n\varphi(n)} \right).
\end{split}
\end{equation*}

Therefore, we have
\begin{equation}\label{lll1}
\begin{split}
\mu(E_{k}) 
&\le \sum_{n = q_{k}+1}^{q_{k+1}}\frac{2}{n\varphi(n)} \le \int_{q_{k}}^{q_{k+1}} \frac{2 dx }{x \varphi(x)} 
= \int_{\log q_{k}}^{\log q_{k+1}} \frac{2 dt}{\varphi(e^{t})}  
\le \frac{2 \log (q_{k+1}/q_k)}{\varphi(q_{k})}. 
\end{split}
\end{equation}
If $\varphi(q_k) q_{k} < q_{k+1}$, then we have
\begin{equation}\label{lll2}
\begin{split}
\mu(E_{k}) 
&\le \frac{2}{\varphi(q_k)} + \sum_{n= 2q_k+1}^{\lceil q_{k+1}/\varphi(q_k) \rceil } \| q_k \theta \| + \sum_{n= \lceil q_{k+1}/\varphi(q_k) \rceil +1 }^{q_{k+1}}  \frac{2}{n \varphi (n)} \\
&\le \frac{2}{\varphi(q_k)} + \left( \left \lceil \frac{q_{k+1}}{\varphi(q_k)} \right \rceil - 2 q_k \right)\| q_k \theta \| 
+ \int_{q_{k+1}/\varphi(q_k)}^{q_{k+1}} \frac{2 dx }{x \varphi(x)} \\
&< \frac{2}{\varphi(q_k)} + \frac{q_{k+1}  \| q_k \theta \| }{\varphi(q_k)} + \int_{\log(q_{k+1}/\varphi(q_k))}^{\log(q_{k+1})} \frac{2dt}{\varphi(e^{t})} \\ 
&< \frac{3}{\varphi(q_{k})} + \frac{2\log \varphi(q_k)}{\varphi(q_{k+1}/\varphi(q_k))} 
 \le \frac{3}{\varphi(q_{k})} + \frac{2\log \varphi(q_k)}{\varphi(q_{k})}. 
\end{split}
\end{equation}
By (\ref{lll1}) and (\ref{lll2}),
$$ \sum_{k=0}^\infty \mu (E_{k}) \le \left( \frac{3}{\log \varphi(1)} + 2 \right) \sum_{k= 0}^\infty \frac{\log \left( \min ( \varphi(q_k), q_{k+1}/q_k) \right) }{\varphi(q_k)} < \infty, $$ 
Therefore, by the Borel-Cantelli Lemma and \eqref{bnek}, the proof of \lq only if ' part is obtained.


Let 
$$ q^*_k = \begin{cases} \max \{ n \ge q_{k} \, | \, n \varphi(n) < q_{k+1} \}, &\text{ if }   q_{k} \varphi(q_{k}) < q_{k+1},  \\
q_{k}, &\text{ if }  q_{k} \varphi(q_{k}) \ge q_{k+1}.\end{cases}$$
By the assumption $\varphi(n) \ge 4$, we have $q_{k} \le q^*_k < q_{k+1}$.

Let 
$$ F_k := \bigcup_{q_{k-1} < n \le q_{k+1}} B\left( n\theta, \frac{1}{n\varphi(n)} \right)$$
and put 
$$ b_{k+1} = \left \lceil \frac{q_{k+1} - q^*_k}{q_{k}} \right \rceil, \qquad  1 \le b_{k+1} \le a_{k+1}$$
and for $0 \le i < b_{k+1}$
$$G_{k,i} := \bigcup_{q_{k+1} - (i+1)q_{k} < n \le q_{k+1} - iq_{k}} B\left( n\theta, \frac{1}{ 8(q_{k+1} - iq_{k}) \varphi(q_{k+1})} \right).$$
Then we have
$$G_k := \bigcup_{i = 0}^{b_{k+1}-1} G_{k,i} \subset F_k.$$

Since 
$$
b_{k+1} < \frac{q_{k+1} - q^*_k}{q_{k}} +1,
$$
each ball in $G_k$ has radius at most 
\begin{align*}
\frac{1}{ 8(q_{k+1} - (b_{k+1}-1) q_{k}) \varphi(q_{k+1})} &\le \frac{1}{ 8 q^*_{k} \varphi(q_{k+1})}
< \frac{1}{ 4 (q^*_{k}+1) \varphi(q^*_{k} +1)} \\
&\le \frac{1}{4 q_{k+1}} < \frac{\|q_k \theta \|}{2}.
\end{align*}
Since the balls in $G_k$ are distanced at least by $\|q_k \theta \|$, 
we have 
$$ \mu(G_{k,i}) = \frac{q_k}{ 8(q_{k+1} - i q_{k}) \varphi (q_{k+1} ) }, \qquad
\mu(G_{k}) = \frac{q_k}{8\varphi(q_{k+1})} \sum_{i = 0}^{b_{k+1}-1} \frac{1}{q_{k+1} - i q_{k}}.
$$
Since $$ \sum_{i=0}^{m} \frac{1}{q - ic} = \frac{1}{q - mc} + \dots + \frac{1}{q}
\ge \int_{q - mc}^{q+c} \frac 1x \frac{dx}{c} = \frac 1c \log \frac{q+c}{q-mc},$$
we have
\begin{equation}\label{ineq0}
\mu(G_{k}) \ge \frac{1}{8\varphi(q_{k+1})} \log \frac{q_{k+1}+q_{k}}{q_{k+1}-(b_{k+1}-1)q_{k}}
\ge 
\frac{1}{8\varphi(q_{k+1})} \log \frac{q_{k+1}+q_{k}}{q_{k}+q^*_k}. 
\end{equation}

If $q^*_{k} = q_{k}$, then $b_{k+1} = a_{k+1}$, thus we have
\begin{equation}\label{ineq2}
\mu(G_{k}) \ge \frac{1}{8\varphi(q_{k+1})} \log \frac{q_{k+1}+q_{k}}{q_{k+1}-(b_{k+1}-1)q_{k}} 
\ge \frac{1}{8\varphi(q_{k+1})} \log \frac{q_{k+1}+q_k}{ q_k + q_{k-1}}.
\end{equation}

If $q_k \varphi(q_k) < q_{k+1}$, 
then we have from $\varphi(q^*_k) \ge 4$ 
\begin{equation}\label{ineq1}
\mu(G_{k}) 
\ge \frac{1}{8\varphi(q_{k+1})} \log \frac{q_{k+1}+q_{k}}{q_{k}+q^*_k}
\ge \frac{1}{8\varphi(q_{k+1})} \log \frac{\varphi(q^*_k)}{2}
\ge \frac{1}{16} \frac{\log \varphi(q^*_k)}{\varphi(q_{k+1})}. 
\end{equation}

\begin{lemma}
If
$$\sum_{k=0}^\infty \frac{ \log ( \min (\varphi (q_k), q_{k+1}/q_k) )}{ \varphi( q_{k} ) } = \infty,$$
then $$\sum_{k=0}^\infty \mu(G_k) = \infty.$$
\end{lemma}

\begin{proof}
By Proposition~\ref{qk} we have
$$\sum_{k=0}^\infty \frac{ \log ( \min (\varphi (q_k), q_{k+1}/q_k) )}{ \varphi( q_{k+1} ) } = \infty.$$
Let 
$$\Delta = \{ k \ge 0 \, | \, q_k \varphi(q_k) < q_{k+1} \}. $$
Then either
$$\sum_{k \in \Delta} \frac{ \log ( \min (\varphi (q_k), q_{k+1}/q_k) )}{ \varphi( q_{k+1} ) } = \sum_{k \in \Delta}  \frac{ \log \varphi (q_k)}{ \varphi( q_{k} ) } = \infty$$
or
$$\sum_{k \in \Delta^c} \frac{ \log ( \min (\varphi (q_k), q_{k+1}/q_k) )}{ \varphi( q_{k+1} ) } = \sum_{k \in \Delta^c} \frac{ \log (q_{k+1}/q_k )}{ \varphi( q_{k} ) } = \infty.$$

If $\sum_{k \in \Delta}  \log \varphi (q_k) / \varphi( q_{k} ) = \infty$, then 
by \eqref{ineq1}, we have
$$ \sum_{k=0}^\infty \mu(G_k) \ge \sum_{k \in \Delta} \mu(G_k) \ge \frac{1}{16} \sum_{k \in \Delta} \frac{\log \varphi(q^*_k)}{\varphi(q_{k+1})} = \infty.$$

Suppose that 
$$\sum_{k \in \Delta^c} \frac{ \log (q_{k+1}/q_k )}{ \varphi( q_{k+1} ) } = \infty.$$ 
If $k -1 \in \Delta$ and $k,k+1,\dots,k+m \in \Delta^c$, then by \eqref{ineq0} and \eqref{ineq2},
\begin{multline*}
\mu(G_{k-1}) + \mu(G_{k}) + \dots + \mu(G_{k+m})
\ge 
 \frac{1}{8\varphi(q_{k})} \log \frac{q_{k}+q_{k-1}}{ q_{k-1} + q^*_{k-1}} \\
+ \frac{1}{8\varphi(q_{k+1})} \log \frac{q_{k+1}+q_k}{ q_k + q_{k-1}}
+ \dots
+ \frac{1}{8\varphi(q_{k+m+1})} \log \frac{q_{k+m+1}+q_{k+m}}{ q_{k+m} + q_{k+m-1}}.
\end{multline*}

For $k-1 \le i \le k+m$, let
\begin{align*}
f(x) &= \frac{1}{8\varphi(q_{i+1})}, & \log (q_{i}+q_{i-1}) \le x < \log (q_{i+1}+q_{i}), \\
g(x) &= \frac1{8\varphi(q_{i+1})},  &\log (q_{i}) \le x < \log (q_{i+1}).
\end{align*}
Then $$f(x) \ge g(x) \ \text{ for } \ \log q_k \le x < \log q_{k+m+1}.$$
Since $k-1 \in \Delta$, we have $q_{k-1} \varphi(q_{k-1}) \le q^*_{k-1} \varphi(q^*_{k-1}) < q_k$.
By the assumption $\varphi(q_{k-1}) \ge 4$, we have $q_{k-1} + q^*_{k-1} < q_k$.
Therefore, we have
\begin{equation*}
\begin{split}
\mu(G_{k-1}) + \mu(G_{k}) + \dots + \mu(G_{k+m})
&= \int_{\log (q_{k-1} + q^*_{k-1})}^{\log (q_{k+m+1}+q_{k+m})} f(x) dx \\
&\ge \int_{\log q_{k}}^{\log q_{k+m+1}} g(x) dx \\
&=\frac{\log (q_{k+1}/q_k)}{8\varphi(q_{k+1})}
+ \dots + \frac{\log (q_{k+m+1}/q_{k+m})}{8\varphi(q_{k+m+1})} . 
\end{split}
\end{equation*}
Hence, 
\begin{equation*}
 \sum_{k=0}^\infty \mu(G_k) \ge \sum_{k \in \Delta^c} \frac{\log (q_{k+1}/q_{k})}{8 \varphi(q_{k+1})} = \infty. \qedhere
\end{equation*}

\end{proof}

Now we estimate $\mu(G_\ell \cap G_k)$, $\ell< k$ by the Denjoy-Koksma inequality (see e.g., \cite{Herman}): Let $T$ be an irrational rotation by $\theta$ and $f$ be a real valued function of bounded variation on the unit interval. Then for any $x$ we have
\begin{equation}\label{Koksma} 
\left| \sum_{n=0}^{q_k -1} f(T^n x) - q_k \int f \, d\mu \right | < \text{\rm var} (f). \end{equation}

For a given interval $I$, by the Denjoy-Koksma inequality (\ref{Koksma}) we have
\begin{equation*}
\# \left\{ 0 \le n < q_{k} \, | \, n\theta \in I \right\} = \sum_{n=0}^{q_{k}-1} 1_{I} (T^n x) < q_{k} \mu (I) + 2.
\end{equation*}
Since $G_{k,i}$ consists of the intervals of centered at $q_k$ orbital points 
 with radius $(8(q_{k+1} - iq_{k}) \varphi(q_{k+1}))^{-1}$, we have for each $i$ 
$$ \mu \left( G_{k,i} \cap I \right) 
< \left( q_{k} \mu(I) + 3 \right) \frac{1}{4(q_{k+1} - i q_{k}) \varphi(q_{k+1})}
= \mu(G_{k,i}) \mu(I) + \frac{3}{q_k} \mu (G_{k,i}).$$

Note $G_\ell$ consists of at most $q_{\ell+1}$ intervals.

Therefore, we have for $k > \ell$
\begin{equation*}
\mu( G_{k,i} \cap G_\ell ) 
< \mu( G_{k,i} ) \mu( G_\ell ) + \frac{3 q_{\ell+1}}{q_k} \mu(G_{k,i}). 
\end{equation*}

Since $G_k = \cup G_{k,i}$ by a disjoint union, we have
\begin{equation*}\begin{split}
\mu( G_{k} \cap G_\ell ) &< \mu( G_{k} ) \mu( G_\ell ) + \frac{3 q_{\ell+1}}{q_k} \mu(G_{k}) \\
&< \mu( G_{k} ) \mu( G_\ell ) + 3 \left(\frac{1}{2}\right)^{\lfloor (k-\ell-1)/2 \rfloor} \mu(G_{k}) \\
&\le \mu( G_{k} ) \mu( G_\ell ) + \frac{6}{2^{(k-\ell)/2}} \mu(G_{k}). 
\end{split}
\end{equation*}

We need a version of Borel-Cantelli lemma (e.g. \cite {Sp}) to go further:
\begin{lemma}\label{Sp}
Let $(\Omega, \mu)$ be a measure space,
let $f_k(\omega)$ $(k=1,2,\dots)$ be a sequence of nonnegative $\mu$-measurable functions, and let $f_k, \varphi_k$ be sequences of real numbers such that
$$ 0 \le f_k \le \varphi_k \le 1 \qquad (k = 1, 2, \dots).$$
Suppose that
$$ \int_\Omega \left( \sum_{m < k \le n} f_k (\omega) - \sum_{m < k \le n} f_k \right)^2 d\mu \le C \sum_{m < k \le n} \varphi_k$$
for arbitrary integers $m$, $n$ ($m<n$). Then
$$ \sum_{1 \le k \le n } f_k (\omega) = \sum_{1 \le k \le n} f_k + O( \Phi^{1/2}(n) \ln^{3/2+\varepsilon} \Phi (n))$$
for almost all $\omega \in \Omega$, where $\varepsilon >0$ is arbitrary and $\Phi(n) = \sum_{1 \le k \le n} \varphi_k$.
\end{lemma}

Put $f_k = \varphi_k = \mu(G_k)$ and $f_k (x) = 1_{G_k} (x)$ in Lemma~\ref{Sp}.
Then we have for any $m<n$
\begin{equation*}
\begin{split}
&\int \left( \sum_{m < k \le n} f_k (\omega) - \sum_{m < k \le n} f_k \right)^2 d\mu \\
&\le 2 \sum_{m < \ell < k \le n} \left( \mu ( G_k \cap G_{\ell} ) - \mu (G_k) \mu (G_\ell) \right)  +  \sum_{m < k \le n} \mu(G_k) \\
&< 2 \sum_{m < k \le n}  \sum_{m < \ell < k} \frac{6}{2^{(k-\ell)/2}} \mu (G_k) + \sum_{m < k \le n} \mu(G_k) 
< \left( \frac{12}{\sqrt{2} -1} +1 \right) \sum_{m < k \le n} \mu (G_k).
\end{split}
\end{equation*}
Therefore, by Lemma~\ref{Sp}, 
if 
$$ \sum_k \mu (G_k) = \infty,$$
then  we have for almost every $x$ 
$$\sum_{k=1}^\infty 1_{G_k} (x) = \infty$$
or 
$$ x \in G_k \subset F_k \text{ infinitely many $k$'s}.$$ 
Hence, we have the proof of Theorem~\ref{thm}.


\section*{Acknowledgments}
The author wish to thank Hitoshi Nakada, Bao-Wei Wang, Jian Xu and Bo Tan for many helpful discussion.

\end{document}